\documentclass[12pt,a4paper]{article}
\title{Krull-Tropical Hypersurfaces.\footnote{{\em MSC}:12.70,14B99. {\em
Key words}: Krull valuations, Tropical geometry, algebraic variety,
max-plus algebras}}
\author{Fuensanta Aroca.\footnote{Partially supported by CONACyT 55084.}}

\usepackage{amsmath,amsfonts,amssymb,latexsym, amscd, euscript}
\usepackage{amsthm}
\newtheorem{lem}{Lemma}
\newtheorem{defin}[lem]{Definition}
\newtheorem{remark}[lem]{Remark}
\newtheorem{thm}[lem]{Theorem}
\newtheorem{prop}[lem]{Proposition}

\newtheorem{cor}[lem]{Corollary}

\usepackage{makeidx}
\usepackage[english]{babel}
\usepackage[latin1]{inputenc}
\usepackage[all]{xy}
\usepackage{multicol}

\usepackage[dvips]{graphicx}

\usepackage{psfrag}

\graphicspath{{./pix/}}
\usepackage{enumerate}
\usepackage{color}
\numberwithin{equation}{section}
\numberwithin{lem}{section}
\newpage
\makeindex

\begin{document}
\maketitle
\newcommand{\End}{End}
\newcommand{\Ext}{Ext}
\newcommand{\HH}{H}
\newcommand{\Hom}{Hom}
\newcommand{\Ker}{Ker}
\newcommand{\Imm}{Im}
\newcommand{\SL}{SL_\CC}
\newcommand{\Tor}{Tor}
\newcommand{\codim}{codim}
\newcommand{\corank}{corank}
\newcommand{\coker}{coker}
\newcommand{\hd}{hd}
\newcommand{\reg}{reg}

\newcommand{\Vop}{{ V_0 \otimes \dots \otimes V_p}}
\newcommand{\Pop}{{\PP^{k_0}\times  \dots   \times \PP^{k_p}}}

\newcommand{\AAa}{{\EuScript A}}
\newcommand{\AAm}{{\mathbb A}}
\newcommand{\CC}{{\mathbb C}}
\newcommand{\NN}{{\mathbb N}}
\newcommand{\KK}{{\mathbb K}}
\newcommand{\Hh}{{\EuScript H}}
\newcommand{\EE}{{\EuScript E}}
\newcommand{\Ki}{{\EuScript K}}
\newcommand{\Ss}{{\EuScript S}}
\newcommand{\FF}{{\EuScript F}}
\newcommand{\II}{{\mathfrak I}}
\newcommand{\MM}{{\mathfrak M}}
\newcommand{\mm}{{\mathfrak m}}
\newcommand{\OO}{{\EuScript O}}
\newcommand{\OOl}{{\OO_\ell}}
\newcommand{\PPii}{{{\mathbb P}^2}}
\newcommand{\PPn}{{{\mathbb P}^n}}
\newcommand{\PP}{{\mathbb P}}
\newcommand{\RR}{{\mathbb R}}
\newcommand{\ZZ}{{\mathbb Z}}
\newcommand{\QQ}{{\mathbb Q}}
\newcommand{\TT}{{\mathbb T}}
\newcommand{\GG}{{\mathbb G}}
\newcommand{\alphaminus}{{\alpha_{\scriptscriptstyle -}}}
\newcommand{\alphaplus}{{\alpha_{\scriptscriptstyle +}}}
\newcommand{\dual}[1]{{#1}^*}
\newcommand{\second}{{\prime\prime}}

\newcommand{\ms}[1]{\mbox{\hspace{#1cm}}}
\newcommand{\fr}[1]{\stackrel{#1}{\longrightarrow}}
\newcommand{\A}{\stackrel{2}{\wedge}}
\newcommand{\x}{\otimes}
\newcommand{\Uu}{{x_0+x_1+x_2}}
\newcommand{\p}{\oplus}
\newcommand{\fd}{\rightarrow}
\newcommand{\V}{V_n^{\vee}}
\newcommand{\ac}[1]{\tilde{#1}}
\newcommand{\Om}{\Omega^1}
\newcommand{\vsp}{\mathop\rightarrow\limits}
\newcommand{\Ho}{\HH^{0}}
\newcommand{\W}{\wedge}
\newcommand{\Pa}[1]{#1_{\alpha}}
\newcommand{\Aa}[1]{#1^{\alpha}}
\newcommand{\Pm}[1]{#1_{\mu}}
\newcommand{\Am}[1]{#1^{\mu}}
\newcommand{\NI}{\not\in}
\newcommand{\ch}{\vee}
\newcommand{\LL}[1]{{\mathcal L}_{#1}} 
\newcommand{\Reg}[1]{{\mathcal C}_{#1}} 
\newcommand{\Rag}[1]{{\mathcal D}_{#1}}
\newcommand{\KKK}{\sl K}
\newcommand{\val}{val}
\newcommand{\ord}{ord}
\newcommand{\inval}{In}
\newcommand{\Inval}{{\mathcal I}n}
\newcommand{\idmax}[1]{{\mathcal J}_{#1}}
\newcommand{\VarTrop}{{\bf TV}}
\newcommand{\ceros}{{\bf V}}
\newcommand{\Trop}{{\mathcal T}}
\newcommand{\VarT}{{\mathcal V}}
\newenvironment{RED}{\color{red}}{\normalcolor}
\newenvironment{BLUE}{\color{blue}}{\normalcolor}
\newenvironment{GREEN}{\color{green}}{\normalcolor}
\begin{abstract}

    The concepts of tropical-semiring and tropical hypersurface, are extended for an arbitrary ordered group.
    Then, we define the tropicalization of a polynomial with coefficients in a Krull-valued field.

    After a close study of the properties of the operator ``tropicalization" we conclude with an extension of
 Kapranov's theorem to algebraically closed fields together with a
    valuation over an ordered group.

\end{abstract}

\section*{Introduction}

The {\em tropical semi-ring} is the set $\TT:=\RR\cup\{\infty\}$
together with the operations $a\oplus b:=\min \{a,b\}$ and $a\odot b
:= a+b$. A {\em tropical hypersurface} is a subset of $\RR^N$
defined by a polynomial with coefficients in $\TT$. A valuation of a
field into the real numbers is used to {\em tropicalize} algebraic
geometry propositions. A naturally real-valued algebraically closed
field is the field of Puiseux series.

Let $\KK$ be an algebraically closed real-valued field.
In \cite{EinsiedlerKapranov:2006} M. Einsieder, M. Kapranov and D. Lind show that the image of an
algebraic hypersurface via a valuation into the reals coincides with
the non-linearity locus of its tropical map.


Valuations into the real numbers are just a particular type of
valuations called {\em classical} (see for example
\cite{Ribenboim:1999}). In 1932 W. Krull extended the classical
definition considering valuations  with values in an arbitrary
ordered group \cite{Krull:1932}. Krull's definition is the one
currently used in most articles and reference texts (see for example
\cite{ZariskiSamuel:1975II, Eisenbud:1995,Spivakovsky:1990}).

Replacing $\RR$ by another totally ordered group $\Gamma$, the
tropical semi-ring $\GG := \Gamma\cup\{\infty\}$ may be defined
naturally. The same happens with the concept of tropical
hypersurface and the tropicalization of a polynomial. A first step
in this direction has been done in \cite{FAroca:2008} where an
example is given.

In this note we extend these concepts and prove some properties of
the tropicalization map. Using these properties we extend the so
called Kapranov's theorem. Our proof is not just an extension of an
existing proof in the classical case but it is essentially
different.

In \cite{EinsiedlerKapranov:2006}, a tropical hypersurface is
defined as the closure in $\RR^N$ of the image, via valuation, of an
algebraic hypersurface. Defining the tropical hypersurface as a
subset of the group of values has the advantage (even when the group
of values is contained in $\RR$) that we do not need to deal with
topological arguments. This idea is already present in
\cite{Katz:2009}.

Sections 1 and 2 are devoted to extend the definitions of tropical
semiring and tropical hypersurface. In sections 3 and 4 we recall
the definition of Krull valuation and extend the definition of
tropicalization and tropical hypersurface of a polynomial with
coefficients in a valued field.

In section 5 we prove that the hypersurface associated to the
tropicalization of a product is the union of the hypersurfaces of
the tropicalization of its factors. Kapranov's theorem in one
variable comes as a consequence of this fact.

Sections 6 and 7 are devoted to finding elements for which the value
of a polynomial evaluated at a point is equal to the image of its
value under the map induced by the tropicalization of the
polynomial.

In section 8 we give the proof of the extension of Kapranov's theorem.

I would like to thank Jesús del Blanco Maraña for answering all my
naive, and not so naive, questions about valuations. I also thank
Martha Takane and Luc\'ia L\'opez de Medrano for fruitful
discussions during the preparation of this note.

\section{Ordered groups, tropical semi-rings and tropical polynomials.}

A {\bf total ordered group} is an abelian group $(\Gamma ,+)$
equipped with a total order such that for all $x, y,z \in\Gamma$ if
$x\leq y$ then $x+z\leq y+z$. For $a>0$ we have $a+a>0+a$, therefore
a total ordered group is torsion free.

The following definition is an extension of a classical definition for the ordered group
$(\RR ,+,\leq)$ \cite{ItenbergMikhalkin:2007,RichterSturmfels:2005,Gathmann:2006}.
\begin{defin}
A total ordered group $(\Gamma ,+,\leq)$ induces an idempotent
semi-ring $\GG:= (\Gamma\cup\{\infty\},\oplus ,\odot)$. Where
\begin{itemize}
    \item $a\oplus b:= \min\{ a,b\}$ and
    $a\oplus\infty := a$  for $a,b\in\Gamma$.
    \item $a\odot b:= a+b$ and
    $a\odot\infty := \infty$  for $a,b\in\Gamma$.
\end{itemize}
This semiring is called the {\bf
min-plus algebra} induced by $\Gamma$ or the {\bf tropical semi-ring}.
\end{defin}

A non-zero Laurent polynomial $F\in \GG [x^*]:=\GG
[x_1,x_1^{-1},\ldots ,x_N,x_N^{-1}]$ is an expression of the form
\begin{equation}\label{polinomio de Laurent tropical}
F=\bigoplus_{\alpha\in {\mathcal E} (F)\subset\ZZ^N} a_\alpha \odot
x^\alpha,\qquad a_\alpha\in\Gamma ,\quad\#{\mathcal E}
(F)<\infty.
\end{equation}
These polynomials are called {\bf tropical polynomials}.

The set of tropical polynomials is a semi-ring with the natural
operations: Given $F$ as above and $G=\bigoplus_{\beta\in {\mathcal
E} (G)\subset\ZZ^N} b_\beta \odot x^\beta$, we define

\[
F\odot G :=\bigoplus_{\eta\in \EE (F)+\EE (G)} \left(
\bigoplus_{\alpha +\beta =\eta} a_\alpha\odot b_\beta\right)\odot x^\eta
\]
and
\[
F\oplus G :=\bigoplus_{\eta\in\EE (F)\cup\EE (G)} a_\eta\odot b_\eta\odot
x^\eta
\]
where $a_\eta :=\infty$ for all $\eta\in\EE (G)\setminus\EE (F)$ and
$b_\eta :=\infty$  for all $\eta\in\EE (F)\setminus\EE (G)$.

\section{Tropical maps and non-linearity locus.}

Let $\GG$ be the min-plus algebra induced by the group $(\Gamma
,\leq )$.

Given $g \in\GG$ and a natural number $k$, we will use the standard
notation
\[
g^k := \overbrace{ g\odot\cdots \odot g
}^{k\,\text{times}}\quad\text{and}\quad g^{-k} =
{\left(g^{-1}\right)}^k;
\]
and, for $\gamma\in\Gamma^N$ and $\alpha\in\ZZ^N$ we will denote
\[
\gamma^\alpha := {\gamma_1}^{\alpha_1}\odot\cdots\odot
{\gamma_N}^{\alpha_N}.
\]

 A tropical polynomial $F=\bigoplus_{\alpha\in {\mathcal E} (F)\subset\ZZ^N} a_\alpha \odot
x^\alpha$
induces a map
$F:\Gamma^N\longrightarrow \Gamma$ given by
  \[
       F: \gamma
            \mapsto
                \bigoplus_{\alpha\in\EE (F)} a_\alpha \odot
                \gamma^\alpha .
    \]
A map induced by a tropical polynomial is called a {\bf tropical
map}.

 For each $\gamma\in\Gamma^N$ there exists at least one
$\alpha\in \EE (F)$ such that $F (\gamma) =
a_\alpha\odot\gamma^\alpha$. The set of $\alpha$'s with this
property will be denoted by $\Rag\gamma (F)$. That is
\begin{equation}\label{definicion de Rag}
    \Rag\gamma (F) := \{ \alpha\in\EE (F)\mid F (\gamma )= a_\alpha \odot\gamma^\alpha\}.
\end{equation}

\begin{defin}
The {\bf hypersurface associated to $F$} is the subset of $\Gamma^N$
given by
\begin{equation}
\label{Definicion con D y C} \VarT (F):=\{\gamma\in\Gamma\mid
\#\Rag\gamma (F)> 1\}.
\end{equation}
\end{defin}

For $\alpha\in\EE (F)$, the restriction $F|_{\{\gamma\in\Gamma^N\mid
\alpha\in\Rag\gamma (F)\}}: \gamma\mapsto
a_\alpha\odot\gamma^\alpha$ is an {\em affine linear} function on $\Gamma^N$. We
say that $F$ defines a {\em piecewise linear} function,
 being the  hypersurface associated to $F$ its {\em non-linearity locus}.

\section{Valuations.}

Let $(\Gamma ,\leq , + )$ be a total ordered group and let
$(\GG,\oplus ,\odot )$ be its min-plus algebra. A  {\bf valuation}
of a field $(\KK ,+_\KK ,\cdot )$ with values in $(\Gamma ,\leq , +
)$ is a surjective map $\val :\KK\longrightarrow \GG$  such that

\begin{enumerate}
    \item $\val (x) =\infty \Leftrightarrow x=0$,
    \item $\val (x\cdot y)=\val (x) \odot \val (y)$ for all $x,y\in \KK$, and
    \item $\val (x+_\KK y)\geq \val (x) \oplus \val (y)$.
\end{enumerate}

We say that $\KK$ has {\bf values} in $\Gamma$. A field together with a valuation is called a {\bf valued field} and
$(\Gamma ,\leq , + )$ is called {\bf the group of values}.

Note that
\begin{itemize}
    \item
    $\val (1)=\val(1\cdot1) =\val (1) \odot \val (1)\stackrel{\Gamma \text{ is torsion free}}{\Longrightarrow} \val (1)=0$.
    \item
    $0=\val ((-1)(-1))= \val (-1) \odot \val (-1)\stackrel{\Gamma \text{ is torsion free}}{\Longrightarrow}\val (-1)=0$.
    \item
    $\val (-b) = \val ((-1)b)=\val (-1)\odot \val (b) =\val (b)$.
\end{itemize}


\begin{lem}\label{Hay mas de un sumando}
 Let $E\subset \KK$ be a finite set. If $\val\left(\sum_{\varphi\in E}\varphi\right)
 >\oplus_{\varphi\in E}\val\varphi$ then the set of elements in $E$ where the valuation attains its
 minimum has at least two elements.
\end{lem}

\begin{proof}

 Let $E_{\min}$ be the subset of $E$
 consisting of elements where the valuation attains its minimum:
 \[
 E_{\min}=\{ \varphi\in E\mid
    \val\varphi =\oplus_{\varphi\in E}\val\varphi\}.
 \]
Suppose that $E_{\min}=\{ a\}$ and set $b:=\sum_{\varphi\in
E\setminus\{ a\}}\varphi$. We have $\val (b)>\val (a)$ and
\[
\val\left(\sum_{\varphi\in E}\varphi\right)
 >\oplus_{\varphi\in E}\val\varphi \Longleftrightarrow \val (a+b)>\val (a).
\]
Then $\val (a) =\val ((a+b)-b)\geq \val (a+b)\oplus \val (b)>\val
(a)$ which is a contradiction.
\end{proof}

\section{The tropicalization.}

Let $(\KK ,\val)$ be a valued field with values in a group $\Gamma$
and let $\GG$ be the min-plus algebra induced by $\Gamma$.

 A non-zero Laurent polynomial in $N$ variables with coefficients in $\KK$,
$f\in\KK [x^*]$, is written in the form:
\begin{equation}\label{polinomio en K de x}
    f= \sum_{\alpha\in\EE (f)\subset \ZZ^N} \varphi_\alpha x^\alpha\qquad \varphi_\alpha\in\KK\setminus\{ 0\}, \quad \#\EE (f)<\infty.
\end{equation}
The polynomial $f$ via the valuation $\val$ induces an element of
$\GG [x^*]$
\[
\Trop f:= \bigoplus_{\alpha\in \EE (f)  \subset \ZZ^N} \val
(\varphi_\alpha)\odot x^\alpha
\]
this polynomial is called {\bf the tropicalization of $f$}.

\begin{remark}\label{El valor de la funcion es menor o igal que la tropicalizacion envaluado en el valor}
    Since $\val (a+b)\geq\val (a)\oplus\val (b)$ and $\val (ab)= \val
    (a)\odot\val (b)$, we have
    \[
    \val (f(x))\geq \Trop f (\val (x))\quad\text{for all}\quad
    x\in\KK^N.
    \]
\end{remark}

Given a  Laurent polynomial in $N$ variables with coefficients in
$\KK$, $f\in\KK [x^*]:=\KK [x_1,x_1^{-1},\ldots ,x_N,x_N^{-1}]$, the
set of {\bf zeroes} of $f$ is defined as
\[
\ceros (f):= \{ x\in {(\KK\setminus\{ 0\})}^N\mid f(x)=0\}.
\]

The {\bf tropical hypersurface} associated to $f$ is the set of
values of $\ceros (f)$. That is:
\[
\VarTrop f:= \val (\ceros (f)).
\]

\begin{prop}\label{El valor de un cero esta en la tropicalizacion}
    Let $f$ be a non-zero polynomial in $\KK
    [x_1,x_1^{-1},\ldots ,x_N,x_N^{-1}]$.
     If $\phi\in\KK^N$ is a zero
    of $f$, then ${\val} \phi$
    is in the hypersurface associated to the tropicalization of $f$.
    That is:
    \[
    \VarTrop f\subset \VarT\Trop f.
    \]
\end{prop}

\begin{proof}

    For $f=\sum_{\alpha\in\EE (f)}\varphi_\alpha x^\alpha$, we have
    \[
    \Trop f =\oplus_{\alpha\in\EE (f)} \val (\varphi_\alpha)\odot
    x^\alpha.
    \]
    Since
    $\sum_{\alpha\in\EE (f)}\varphi_\alpha\phi^\alpha=0$,
     by lemma \ref{Hay mas de un
    sumando}, the set
    \[
    E_{\min} :=\{\alpha_0 \in\EE (f)\mid \val (\varphi_{\alpha_0}\phi^{\alpha_0}) = \oplus_{\alpha\in\EE (f)}
    \val (\varphi_{\alpha}\phi^{\alpha})\}
    \]
    has at least two elements.

    Now $\val (\varphi_{\alpha}\phi^{\alpha})= \val (\varphi_{\alpha})\odot {(\val \phi)}^{\alpha}$, then $E_{\min} =
    \Rag{\val\phi} (\Trop f)$ and
    we have the result.
\end{proof}

\section{The tropicalization of a product.}
The map $\Trop : \KK [x_1,\ldots ,x_N]\longrightarrow \GG
[x_1,\ldots ,x_N]$ may not preserve sum nor product.
 Anyhow, the tropical variety of the product may be described.
 \begin{lem}\label{Si elijo los mas pequennos su suma esta}
    Let $\KK$ be a valued field and let $\Gamma$ be its group of
    values. Given $\omega\in\RR^N$ of rationally independent
    coordinates, $f,g\in\KK [x^*]$ and $\gamma\in\Gamma^N$; set
    $\alpha_0\in\Rag\gamma (\Trop f)$ and $\beta_0\in\Rag\gamma
    (\Trop g)$ such that
    \begin{equation}\label{alpha0 es el minimo}
\omega\cdot\alpha_0=\min_{\alpha\in\Rag\gamma (\Trop
f)}\omega\cdot\alpha
\quad\text{and}\quad\omega\cdot\beta_0=\min_{\beta\in\Rag\gamma
(\Trop g)}\omega\cdot\beta.
\end{equation}
Set $\eta_0:=\alpha_0+\beta_0$. We have:
\[
\eta_0\in\Rag\gamma (\Trop (fg))\quad\text{and}\quad
\omega\cdot\eta_0=\min_{\eta\in\Rag\gamma (\Trop
    fg)}\omega\cdot\eta.
\]
 \end{lem}

\begin{proof}
    Set $f=\sum_{\alpha\in\EE (f)} \varphi_\alpha x^\alpha$ and $g=\sum_{\beta\in\EE (g)} \varphi'_\beta x^\beta$. Then
\[
fg=\sum_{\eta\in\EE (f)\cup\EE (g)}\left(\sum_{\alpha +\beta =\eta} \varphi_\alpha\varphi'_\beta\right) x^\eta.
\]


 By (\ref{alpha0 es el minimo}),
 we have
\begin{equation}\label{en eta0 se alcanza el omegamin}
    \omega\cdot\eta_0=\min_{\eta\in\Rag\gamma (\Trop
    f)+\Rag\gamma (\Trop g) }\omega\cdot\eta.
\end{equation}

Since $\alpha_0\in\Rag\gamma (\Trop f)$ and $\beta_0\in\Rag\gamma (\Trop g)$, by definition (\ref{definicion de Rag}), we have
\begin{equation}\label{alpha0 y beta0 estan el Raggamma}
    \left\{
    \begin{array}{lcc}
        &\val (\varphi_{\alpha_0})\odot\gamma^{\alpha_0} \leq \val (\varphi_{\alpha})\odot\gamma^{\alpha},
            & \forall \alpha\in\EE (f)\\
        \text{and}
            & &\\
        &\val (\varphi'_{\beta_0})\odot\gamma^{\beta_0} \leq \val (\varphi'_{\beta})\odot\gamma^{\beta},
            & \forall \beta\in\EE (g).
    \end{array}
    \right.
\end{equation}
and
\begin{equation}\label{alpha0 y beta0 estan el Raggamma y alpha y beta no}
    \left\{
    \begin{array}{lcc}
        &\val (\varphi_{\alpha_0})\odot\gamma^{\alpha_0} < \val (\varphi_{\alpha})\odot\gamma^{\alpha},
            & \forall \alpha\in\EE (f)\setminus\Rag\gamma (\Trop f)\\
        \text{and}
            & &\\
        &\val (\varphi'_{\beta_0})\odot\gamma^{\beta_0} < \val (\varphi'_{\beta})\odot\gamma^{\beta},
            & \forall \beta\in\EE (g)\setminus\Rag\gamma (\Trop g).
    \end{array}
    \right.
\end{equation}

 Let $\alpha\in\EE (f)$ and
$\beta\in\EE (f)$ be such that $\eta_0=\alpha +\beta$. If
$\alpha_0\neq\alpha$ then either $\omega\cdot\alpha
<\omega\cdot\alpha_0$ or $\omega\cdot\beta <\omega\cdot\beta_0$.
Then, by (\ref{alpha0 es el minimo}), $\alpha\notin\Rag\gamma (\Trop
f)$ or $\beta\notin\Rag\gamma (\Trop g)$, and then
\begin{equation}\label{menos estricto con los alpha0}
    \left\{
    \begin{array}{lc}
            &\alpha +\beta =\eta_0\\
        \text{and}
            &\\
            &\alpha\neq\alpha_0
    \end{array}
    \right.
    \Rightarrow
    \left\{
    \begin{array}{lc}
            &\val (\varphi_{\alpha_0} )\odot\gamma^{\alpha_0}<\val (\varphi_\alpha )\odot\gamma^\alpha\\
        \text{or}
            &\\
            &\val (\varphi'_{\beta_0} )\odot\gamma^{\beta_0}<\val (\varphi'_\beta )\odot\gamma^\beta.
    \end{array}
    \right.
\end{equation}
Inequalities (\ref{menos estricto con los alpha0}) together with
(\ref{alpha0 y beta0 estan el Raggamma}) give
\[
    \left\{
    \begin{array}{lc}
            &\alpha +\beta =\eta_0\\
        \text{and}
            &\\
            &\alpha\neq\alpha_0
    \end{array}
    \right.
    \Rightarrow
\val (\varphi_{\alpha_0}\varphi'_{\beta_0})\odot\gamma^{\eta_0}<\val
(\varphi_{\alpha}\varphi'_{\beta})\odot\gamma^{\eta_0}.
\]
Therefore, by lemma \ref{Hay mas de un sumando},
\begin{equation}\label{el valor es el que tiene que ser}
\val \left(\sum_{\alpha +\beta
=\eta_0}\varphi_\alpha\varphi'_\beta\right)=\val
(\varphi_{\alpha_0}\varphi'_{\beta_0}).
\end{equation}
Inequalities (\ref{alpha0
y beta0 estan el Raggamma}) give
\begin{equation}\label{la desigualdad en la suma}
\val \left(\varphi_{\alpha_0}\varphi'_{\beta_0}\right)\odot\gamma^{\eta_0}\leq
\val \left(\varphi_{\alpha}\varphi'_{\beta}\right)\odot\gamma^{\alpha+\beta},\quad\forall \alpha\in\EE (f),\beta\in\EE (g).
\end{equation}
Equality (\ref{el valor es el que tiene que ser}) together with
(\ref{la desigualdad en la suma}) give
\[
\val \left(\sum_{\alpha +\beta
=\eta_0}\varphi_\alpha\varphi'_\beta\right)\odot\gamma^{\eta_0}\leq
\val \left(\sum_{\alpha +\beta
=\eta}\varphi_\alpha\varphi'_\beta\right)\odot\gamma^\eta,\,\forall\eta\in\EE
(f)+\EE (g).
\]
In other words:
\begin{equation}\label{eta0 esta en rag de gf}
\eta_0\in\Rag\gamma (\Trop fg)
\end{equation}
and
\begin{equation}\label{El valor de la tropicalizacion en gamma en funcion de alpha0y bata0}
\Trop fg (\gamma )=\val \left(\sum_{\alpha +\beta
=\eta_0}\varphi_\alpha\varphi'_\beta\right)\odot\gamma^{\eta_0}\
 =  \val
(\varphi_{\alpha_0}\varphi'_{\beta_0})\odot\gamma^{\eta_0}.
\end{equation}
By (\ref{alpha0 y beta0 estan el Raggamma y alpha y beta no}) and
(\ref{El valor de la tropicalizacion en gamma en funcion de alpha0y bata0}), we have
\begin{equation}\label{rag del producto contenido en suma de rags}
\Rag\gamma (\Trop fg)\subset\Rag\gamma (\Trop f)+\Rag\gamma (\Trop g)
\end{equation}
(\ref{en eta0 se alcanza el omegamin}),  (\ref{eta0 esta en rag de gf}) and
(\ref{rag del producto contenido en suma de rags}) give
\begin{equation}
    \omega\cdot\eta_0=\min_{\eta\in\Rag\gamma (\Trop
    fg)}\omega\cdot\eta.
\end{equation}
\end{proof}

\begin{prop}
The hypersurface associated to the tropicalization of a finite
product of polynomials is equal to the union of the hypersurfaces
associated to the tropicalization of each polynomial. That is
\[
\VarT \Trop (fg)=\VarT (\Trop f )\cup \VarT (\Trop g).
\]
\end{prop}
\begin{proof}
Take $\omega\in\RR^N$ of rationally independent coordinates. Set
$\alpha_0\in\Rag\gamma (\Trop f)$ and $\beta_0\in\Rag\gamma (\Trop
g)$ such that
\[
\omega\cdot\alpha_0=\min_{\alpha\in\Rag\gamma (\Trop
f)}\omega\cdot\alpha\quad\text{and}\quad
\omega\cdot\beta_0=\min_{\beta\in\Rag\gamma (\Trop
g)}\omega\cdot\beta.
\]
    Now take $\alpha_1\in\Rag\gamma (\Trop f)$ and $\beta_1\in\Rag\gamma
(\Trop g)$
 such that
\[
(-\omega)\cdot\alpha_1=\min_{\alpha\in\Rag\gamma (\Trop
f)}(-\omega)\cdot\alpha\quad\text{and}\quad
(-\omega)\cdot\beta_1=\min_{\beta\in\Rag\gamma (\Trop
g)}(-\omega)\cdot\beta.
\]
By lemma \ref{Si elijo los mas pequennos su suma esta} we have
$\eta_0:=\alpha_0+\beta_0,\eta_1:=\alpha_1+\beta_1\in\Rag\gamma
(\Trop (fg))$. And

\begin{equation}
   \omega\cdot\eta_0=\min_{\eta\in\Rag\gamma (\Trop
    fg)}\omega\cdot\eta.\quad\text{and}\quad \omega\cdot\eta_1=\max_{\eta\in\Rag\gamma (\Trop
    fg)}\omega\cdot\eta.
\end{equation}
Now
\[
\left\{
\begin{array}{cc}
        & \gamma\in\VarT\Trop f\\
    \text{or}
        &\\
        &\gamma\in\VarT\Trop g
\end{array}
\right. \Leftrightarrow
\left\{
\begin{array}{cc}
        & \alpha_0\neq\alpha_1\\
    \text{or}
        &\\
        &\beta_0\neq\beta_1
\end{array}
\right. \Leftrightarrow \eta_0\neq\eta_1\Leftrightarrow \gamma\in\VarT\Trop fg.
\]
\end{proof}

\begin{cor}\label{Kapranov en una variable}
    Let $(\KK, \val)$ be an algebraically closed valued field. For $N=1$ and $f\in\KK [x]$ we have
    \[
    \VarT\Trop f =\VarTrop f.
    \]
\end{cor}
\begin{proof}
$f=\prod_{a\in\ceros (f)} (x-a)$ then $\VarT\Trop f= \cup_{{a\in\ceros (f)}}\VarT\Trop (x-a)=\{ \val (a)\mid {a\in\ceros (f)}\}$.
\end{proof}

\section{Valuation ring and residue field.}

The set
\[
{\sl A}_{\val} :=\{ a\in\KK\mid \val (a)\geq 0\}
\]
is a ring called the {\bf valuation ring}. The valuation ring has only one maximal ideal given by
\[
\mm_{\val} :=\{ a\in\KK\mid \val (a)> 0\},
\]
whose group of units is given by:
\[
{\sl U}_{\val} :=\{ a\in\KK\mid \val (a)= 0\}.
\]
Its {\bf residue field} is defined as
\[
{\sl R}_{\val} := {\sl A}_{\val}/\mm_{\val}.
\]
There is a natural map
\begin{equation}\label{aplicacion natural en el campo residual}
\begin{array}{ccc}
    {\sl A}_{\val}
        &\longrightarrow
            & {\sl R}_{\val}\\
    \varphi
        &\mapsto
            &\bar{\varphi}.
\end{array}
\end{equation}

\begin{lem}\label{el campo residual es infinito}
If $\KK$ is algebraically closed, then its residue field is
algebraically closed.
\end{lem}
\begin{proof}
    Given $P(x)\in {\sl R}_{\val} [x]\setminus {\sl R}_{\val}$ let $Q(x)\in {\sl A}_{\val} [x]\setminus {\sl A}_{\val}$
    be a pre-image of
    $P(x)$ via the map (\ref{aplicacion natural en el campo residual}). Since ${\sl A}_{\val}\subset\KK$, the polynomial
    $Q$ has a root $k\in\KK$.

Write $Q= \sum_{j=0}^d u_j x^j\in {\sl U}_{\val}$ with $u_0,u_d\neq
0$. We have
\[
 \val (u_j x^j)=  j\,\val (k).
\]
Since $\sum_{j=0}^d u_j k^j=0$,  by lemma \ref{Hay mas de un
sumando}, there exists $j\neq j'$ such that $j\,\val (k)=j'\,\val
(k)$. Then, $\val (k)= 0$ or $\val (k)=\infty$. This implies that
$k$ is an element of ${\sl A}_{\val}$.

 The image of $k$,  via the map (\ref{aplicacion natural en el campo residual}), is a
root of $P$.
\end{proof}

\begin{remark} As a consequence of lemma \ref{el campo residual es
infinito} we have: If $\KK$ is algebraically closed and the
valuation is not trivial, ${\sl R}_{\val}$ is infinite.
\end{remark}

\section{The value of the evaluation of a polynomial at a point.}
As we noted in remmark \ref{El valor de la funcion es menor o igal que la tropicalizacion envaluado en el valor}, we have
\[
\val (f(x))\geq \Trop f (\val (x))\quad\text{for all}\quad
x\in\KK^N,
\]
in this section we will see that, for each value, there exist elements for which the equality holds.

\begin{lem}\label{elementos que no dan no cero en el campo residual}
Let $f_1,\ldots ,f_k$ be a finite set of non-zero Laurent
polynomials in $N$ variables with coefficients in ${\sl R}_{\val}$.
There exists an $N$-tuple of non-zero elements ${\sl r}\in {({\sl
R}_{\val}\setminus\{ 0\})}^N$ such that $f_i ({\sl r})$ is non-zero
for each $i\in 1,\ldots ,k$.
\end{lem}
\begin{proof}

Set $g:= \prod_{i=1}^k f_i$, $g$ is a Laurent polynomial
\[
g:=\sum_{\alpha=(\alpha_1,\ldots ,\alpha_N)\in\Lambda\subset \ZZ^N}
r_\alpha a^\alpha,\quad r_\alpha\in {\sl R}_{\val}
,\,\#\Lambda<\infty
\]
set $\beta := (1,\ldots ,1)-(\min_{\alpha\in\Lambda}\alpha_1,\ldots ,\min_{\alpha\in\Lambda}\alpha_N)$. We have $x^\beta g\in <x^{(1,\ldots ,1)}> {\sl R}_{\val} [x]$.

The set of zeroes of $f:=x^\beta g-1$ is a hypersurface of ${{\sl R}_{\val}}^N$ that doesn't intersect the coordinate hyperplanes. Since ${\sl R}_{\val}$ is an algebraically closed field (lemma
\ref{el campo residual es
 infinito}), by the Nullstellensatz (see for example \cite[A6.P1]{Shafarevich:1994}), there exists a point ${\sl r}\in {({\sl R}_{\val}\setminus\{ 0\})}^N$ where $f$  vanishes.

We have:
\[
f({\sl r})={\sl r}^\beta g({\sl r})-1=0 \Rightarrow {\sl r}^\beta
\prod_{i=0}^r f_i ({\sl r})=1 \Rightarrow f_i ({\sl r})\neq 0\,
\forall ,i=1\ldots k.
\]
\end{proof}

\begin{lem}\label{hay una nupla de unidades}
Let $f_1,\ldots ,f_k$ be a finite set of Laurent polynomials in $N$
variables with coefficients in ${\sl A}_{\val}$. If at least one on
the coefficients  of each $f_i$ is a unit, then there exists an
$N$-tuple of units $u\in {{\sl U}_{\val}}^N$ such that $f_i (u)$ is
a unit for each $i\in \{1,\ldots ,k\}$.
\end{lem}
\begin{proof}

 Let $\bar{f_i}$ be the image of $f_i$ in ${\sl R}_{\val} [x^*]$ via the
 natural morphism. That is
\[
\begin{array}{cccc}
    \Phi :
        &{\sl A}_{\val} [x^*]
            &\longrightarrow
                &{\sl R}_{\val} [x^*]\\
        &\sum_{\alpha}\varphi_{\alpha} x^\alpha
            &\mapsto
                & \sum_{\alpha}\bar{\varphi_{\alpha}} x^\alpha
\end{array}
\]
where $\bar{\varphi_\alpha}$ is the image of $\varphi_\alpha$ via
the map (\ref{aplicacion natural en el campo residual}).

Since at least one of the coefficients is a unit
 $\bar{f_i}$ is not zero. By lemma \ref{elementos que no dan no cero en el campo residual},
 there exists an $N$-tuple of non-zero elements ${\sl r}\in {({\sl R}_{\val}\setminus\{ 0\})}^N$ such
 that $\bar{f_i} ({\sl r})$ is non-zero for each $i\in \{1,\ldots ,k\}$. Take $x\in {{\sl A}_{\val}}^N$ such that $\bar{x}= {\sl r}$ via the natural map (\ref{aplicacion natural en el campo residual}).

We have $x\in {{\sl U}_{\val}}^N$ and $\Phi (f_i(x))= \bar{f_i}({\sl r})\neq 0$ implies $f_i(x)\in {\sl U}_{\val}$.
\end{proof}

\begin{prop}\label{hay una Nupla con un valor dado para la que val y f conmutan}
Let $f_1,\ldots ,f_k$ be   Laurent polynomials in $N$ variables with
coefficients in $\KK$. Given an $N$-tuple $\gamma\in\Gamma^N$ there
exists $x\in\KK^N$ such that
\[
\val (x)=\gamma\quad\text{and}\quad  \val (f_i(x))=\Trop f_i (\val
(x))
\]
for all $i\in\{ 1,\ldots ,k\}$.
\end{prop}

\begin{proof}

Take $\phi\in\KK^N$ and $\psi_i\in\KK$ such that $\val\phi=\gamma$
and $\val\psi_i=\Trop f_i (\gamma )$. Set
\[
g_i (x_1,\ldots ,x_N) := \frac{1}{\psi_i} f_i (\phi_1 x_1,\ldots
,\phi_N x_N)
\]
we have $\Trop g_i (0,\ldots ,0) =0$.

Write $g_i=\sum_\alpha \varphi_{i,\alpha} x^\alpha$, we have
 $\Trop g_i (0,\ldots ,0)= \bigoplus_\alpha \val (\varphi_{i,\alpha})=0$.
 Then, for each $i$, there exists $\alpha_0^{(i)}$ such that $\val (\varphi_{(i,\alpha_0^{(i)})})=0$ and
 $\val (\varphi_{i,\alpha})\geq 0$ for all
 $\alpha$. That is $g_i\in {\sl A}_{\val} [x]$ and one of the coefficients is a unit. By lemma \ref{hay una nupla de unidades}, there exists
 $u=(u_1,\ldots ,u_N)\in {{\sl U}_{\val}}^N$ such that $\val (g_i(u))=0$. Then
\[
\val \left(\frac{1}{\psi_i} f_i (\phi_1 u_1,\ldots ,\phi_N
u_N)\right)=0 \Rightarrow \val (f_i (\phi_1 u_1,\ldots ,\phi_N
u_N))=\val (\psi_i )= \Trop f_i (\gamma ).
\]

Since $\val ((\phi_1 u_1,\ldots ,\phi_N
u_N))=\gamma$, we have the result.

\end{proof}

\section{The theorem}
Now we are ready to extend the theorem proved by Einsieder, Kapranov
and Lind.

\begin{thm} Let $\KK$ be an algebraically closed valued field.
The  tropical hypersurface associated to a polynomial $f\in\KK
[x^*]$ is the hypersurface associated to the tropicalization of $f$.
That is,
\[
\VarTrop f=\VarT\Trop f.
\]

\end{thm}

\begin{proof}
The inclusion $\VarTrop f\subset\VarT\Trop f$ is just proposition
\ref{El valor de un cero esta en la tropicalizacion}.

To see the other inclusion:\\
Given $\gamma\in\VarT\Trop f$ we want to see that there exists $\phi
=(\phi_1,\ldots ,\phi_N)$ such that $\val\phi =\gamma$ and
$f(\phi)=0$.

$\gamma\in\VarT\Trop f$ if and only if there exist
$\alpha^{(0)}\neq\alpha^{(1)}\in\Rag\gamma (\Trop f)$. The vector
$\alpha^{(0)}$ is different from $\alpha^{(1)}$ if and only if one
of the coordinates is different. Let's suppose that
${\alpha^{(0)}}_N\neq {\alpha^{(1)}}_N$. Write $f$ as in
(\ref{polinomio en K de x}) and set $\Lambda :=\{\alpha_N\in\ZZ\mid
\alpha\in\EE (f)\}$. The polynomial $f$ may be rewritten in the form
\[
f= \sum_{i\in\Lambda}h_i(x_1,\ldots ,x_{N-1}){x_N}^i\,\,\text{where}\, h_i=\sum_{(\beta ,i)\in\EE (f)}\varphi_{(\beta ,i)}x^{(\beta ,0)}.
\]

Write $\gamma=(\mu ,\eta)\in \Gamma^{N-1}\times\Gamma$, and choose $y\in\KK^{N-1}$ such that $\val y=\mu$ and  $\val (h_i (y))= \Trop h_i (\mu)$ (proposition \ref{hay una Nupla con un valor dado para la que val y f conmutan}). Set
\[
g:=\sum_{i\in\Lambda}h_i(y){x_N}^i\in\KK [x_N].
\]
We have
\[
\begin{array}{ll}
    \Trop f (\gamma)
        &= \oplus_{\alpha\in\EE (f)}\val (\varphi_\alpha)\odot\gamma^\alpha \\
        &=  \oplus_{i\in\Lambda}\left(\oplus_{(\beta ,i)\in\EE (f)}\val (\varphi_{(\beta ,i)})\odot\mu^\beta\right)\odot\eta^i\\
        &= \oplus_{i\in\Lambda} \Trop h_i (\mu) \odot\eta^i\\
        & = \oplus_{i\in\Lambda} \val (h_i (y)) \odot\eta^i\\
        & = \Trop g(\eta).
\end{array}
\]

Write $\alpha^{(k)}=(\beta^{(k)},j^{(k)})\in \ZZ^{N-1}\times\ZZ$, $k=0,1$. We have
\[
\begin{array}{ll}
     \Trop g(\eta)
        &= \val
        (\varphi_{\alpha^{(k)}})\odot\gamma^{\alpha^{(k)}}
        = \val
        (\varphi_{(\beta^{(k)},j^{(k)})})\odot\mu^{\beta^{(k)}}\odot
        \eta^{j^{(k)}}\\
        & = \Trop h_i (\mu )\odot \eta^{j^{(k)}} = \val
(h_i(y))\odot \eta^{j^{(k)}}.
\end{array}
\]
Since $j^{(0)}\neq j^{(1)}$, the element $\eta\in\Gamma$ is in the variety $\VarT\Trop g$, then, by proposition
\ref{Kapranov en una variable}, there exists $z\in\KK$ such that $\val z=\eta$ and $g(z)=0$.

We have $\phi:= (y,z)\in\KK^N$, $\val (y, z)=\gamma$ and $f(y,z)=g(z)=0$.
\end{proof}

\def\cprime{$'$}

\end{document}